\documentclass[12pt,leqno,fleqn]{amsart}  
\usepackage{amsmath,amstext,amsthm,amssymb,amsxtra}
\usepackage[top=1.5in, bottom=1.5in, left=1.25in, right=1.25in]	{geometry}

\usepackage{tikz}
\usetikzlibrary{arrows}

\usepackage{mathtools}
\mathtoolsset{showonlyrefs,showmanualtags}

\usepackage{hyperref} 
\hypersetup{
    colorlinks=true,       
    linkcolor=blue,          
    citecolor=magenta,        
    filecolor=magenta,      
    urlcolor=cyan           
}

\usepackage[msc-links]{amsrefs}

\theoremstyle{plain} 
\newtheorem{lemma}[equation]{Lemma} 
 
\newtheorem{theorem}[equation]{Theorem} 
\newtheorem{corollary}[equation]{Corollary} 

\newtheorem{priorResults}{Theorem}

\theoremstyle{definition}

\theoremstyle{remark}

\numberwithin{equation}{section}

\title[Sparse Endpoint Bounds for Bochner-Riesz Multipliers] {Sparse Endpoint Estimates for Bochner-Riesz \\  Multipliers on the Plane}
\author{Robert Kesler \and Michael T. Lacey}   

\address{ School of Mathematics, Georgia Institute of Technology, Atlanta GA 30332, USA}
\email {rkesler6@math.gatech.edu}
\email {lacey@math.gatech.edu}

\thanks{Research supported in part by grant  from the US National Science Foundation, DMS-1600693 and the 
Australian Research Council ARC DP160100153.}

\begin{document}

\begin{abstract} For $ 0< \lambda < \frac{1}2$, let $ B _{\lambda }$ be the Bochner-Riesz multiplier of index $ \lambda $ on the plane.  Associated to this multiplier is the critical index $1 <  p_ \lambda =  \frac{4} {3+2 \lambda } < \frac{4}3$.  We prove a sparse bound for $ B _{\lambda }$ with indices $ (p_ \lambda , q)$, where $ p_ \lambda ' < q < 4$.   This is a further quantification of the  endpoint  
weak $L^{p_ \lambda}$ boundedness of  $ B _{\lambda }$,  due to Seeger.  
Indeed, the sparse bound immediately implies new endpoint weighted weak type estimates for weights in $ A_1 \cap RH _{\rho }$, where 
$ \rho > \frac4 {4 - 3 p _{\lambda }}$. 
\end{abstract}

	\maketitle

\section{Introduction} 

We establish a range of sparse bounds for the Bochner-Riesz multipliers in the plane, at the critical index of $ p $.  
The Bochner-Riesz multipliers in the plane are defined through the Fourier transform by 
\begin{equation}\label{e:BR}
\mathcal F (B _{\lambda } f ) (\xi ) = (1- \lvert  \xi \rvert ^2  )_+ ^{\lambda } \mathcal F f (\xi ), \qquad 0< \lambda < \tfrac12. 
\end{equation}
The case of $ \lambda = \frac12$ is the critical case in dimension $ d=2$, where $ B _{\frac1{2}}$ 
maps $ L ^{p}$ to itself for all $ 1< p < \infty $, and is of weak-type $ (1,1)$.   
For $ 0< \lambda < \frac1{2}$, the critical index of $ p_ \lambda $ is given by 
\begin{equation}\label{e:pdel}
 {p_ \lambda }= \tfrac{4}{3+2 \lambda} . 
\end{equation}
The Carleson-Sj\"olin theorem gives the $L^p$ boundedness of the $B_\lambda$.   

\begin{priorResults} \cite{MR0361607}  For $ 0< \lambda < \tfrac1{2}$, there holds 
\begin{equation}\label{e:CS}
\lVert B _{\lambda }\rVert _{p\to p} < \infty , \qquad 
 p_ \lambda < p <  p_ \lambda ' = \frac{p_ \lambda } {p_ \lambda -1}.  
\end{equation}
\end{priorResults}

We recall the notion of a 
sparse bound, which  are a particular quantification of the (weak) $ L ^{p}$-bounds for an operator. They in particular immediately imply weighted and vector-valued inequalities.   
Given a sublinear operator $ T$, and $ 1\leq r, s < \infty$, we set 
$ \lVert T \,:\, (r,s)\rVert$ to be the infimum over constants $ C$ so that for  all measurable, bounded, and  compactly supported functions $ f, g$, 
\begin{equation}\label{e:SF}
\lvert  \langle T f, g \rangle \rvert \leq C \sup  \sum_{Q\in \mathcal S} 
 \langle f \rangle _{Q,r} \langle g \rangle _{Q,s} \lvert  Q\rvert . 
\end{equation}
We use throughout the notation $ \langle f \rangle _{Q,r} = [ \lvert  Q\rvert ^{-1} \int _{Q} \lvert  f\rvert ^{r} \;dx ] ^{\frac{1}r}  $. 
The supremum above is formed over all \emph{sparse} collections $\mathcal S$ of cubes.  
A collection $ \mathcal S$ is sparse if for each $ Q\in \mathcal S$ there is a set $ E_Q \subset S$ 
so that (a) $ \lvert  E_Q\rvert \geq \frac1{10}  \lvert  Q\rvert $, and (b) the sets $ \{E_Q \;:\; Q\in \mathcal S\}$ are pairwise disjoint. 
It is essential that the sparse form be allowed to depend upon $ f $ and $ g$. But the point is that the sparse form itself varies over a class of operators with very nice properties.

The sparse variant of the Carleson-Sj\"olin bounds has been studied in \cites{160506401,170509375}. 

\begin{priorResults} \cite{170509375}  \label{t:SBR}
For $0< \lambda < \frac{1}2$, let $R _ \lambda $ be the open rhombus with vertices 
\begin{equation*}
(\tfrac1{p_ \lambda }, \tfrac1{ p _{\lambda }'}),\ 
(\tfrac {1+ 6 \lambda }4, \tfrac1{ p _{\lambda }'}),\ 
(\tfrac1{ p _{\lambda }}, \tfrac {1+ 6 \lambda }4),\ 
(\tfrac1{p_ \lambda ' }, \tfrac1{ p _{\lambda }}). 
\end{equation*}
There holds 
\begin{equation*}
\lVert B _{\lambda} : (r,s)\rVert < \infty , \qquad   (\tfrac1{r}, \tfrac1{s})\in R _{\lambda }
\end{equation*}
Moreover, the inequality above fails for $\frac1r+\frac 1s \geq 1$, with $(\frac 1r, \frac 1s)$ \textbf{not} in the closure of $R_\lambda$. 
\end{priorResults}

\begin{figure}
\begin{tikzpicture}[scale=4] 
\draw[thick,->] (-.2,0) -- (1.2,0) node[below] {$ \frac 1 r$};
\draw[thick,->] (0,-.2) -- (0,1.2) node[left] {$ \frac 1 s$};
\draw[densely dotted] (.75,.25) --  (.9,.1) 
-- (.9,.6) 
-- (.6,.9) 
--  (.1,.9) 
-- (.25,.75) ;
 
\draw[o-o]  (.25,.75) -- (.75,.25); 
\draw (.25,.05) -- (.25,-.05) node[below] {$ \tfrac 14$};
\draw (.75,.05) -- (.75,-.05) node[below] {$ \tfrac 34$};
\draw (.05,.25) -- (-.05,.25) node[left] {$ \tfrac 14$};
\draw (.05, .75) -- (-.05, .75) node[left] {$ \tfrac 34$};
\draw (.9,.05) -- (.9, -.05) node[below] {$ \tfrac 1 {p _{\lambda }}$};
\draw[loosely dashed] (0,1) -- (1.,1.) node[above] {$ (1,1)$} -- (1.,0); 
\draw[line width=.4mm, black] (.9,.1) -- (.9,.27); 
\draw (.9,.1) circle (.7pt);
\draw (.9, .27) circle (.7pt);
\draw[line width=.4mm, black] (.1,.9) -- (.27, .9);
\draw (.1, .9) circle (.7pt);
\draw (.27, .9) circle (.7pt);
\draw (0.5,.75) node {$ R _{\lambda }$}; 
\end{tikzpicture}

\caption{ The  solid diagonal line indicates the Carleson-Sj\"olin bounds. The dotted rhombus $R_ \lambda$ is the open region for which a sparse bound holds for $ B _{\lambda }$, see \cite{170509375}.  The solid vertical line above $ \tfrac1{p _{\lambda }}$ indicates the 
range of sparse bounds proved in this paper from which the solid horizontal line is obtained as an immediate consequence. We leave it as a question if $ B _{\lambda }$ satisfes a sparse bound for $ (p _{\lambda }, s)$, where $ \tfrac{1}4 \leq \tfrac{1}s < \tfrac{1+ 6 \lambda } 4 $.} 
\label{f:}
\end{figure}
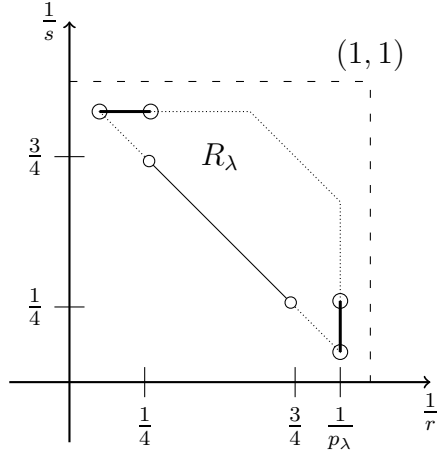

The rhombus $R _ \lambda$ is pictured in Figure \ref{f:}. 
At the endpoint in the Carleson-Sj\"olin theorem,  a weak-type estimate holds, as was proved by Seeger \cite{MR1405600},  extending the work of  Christ \cite{MR1008443}.

\begin{priorResults} 
For $ 0< \lambda < \tfrac1{2}$, there holds 
\begin{equation}\label{e:weak}
\lVert B _{\lambda }\rVert _{p_ \lambda \to (p _{\lambda }, \infty )} < \infty . 
\end{equation}
\end{priorResults}

Our main result is a sparse quantification of this last result.  

\begin{theorem}\label{BP}
Let  $ 0 < \lambda  < \frac{1}{2}$, and $p_\lambda$ as in \eqref{e:pdel},   these sparse bounds hold. 
\begin{equation*}
\lVert  B_{\lambda}: (p_\lambda, q) \rVert < \infty , \qquad   4 < q  < p'_\lambda = \tfrac{p_\lambda}{p_\lambda -1}  . 
\end{equation*}
\end{theorem}
An immediate consequence of Theorem \ref{BP} is that $||B_\lambda : (q, p_\lambda) || <\infty$ for the same range of $q$ and $\lambda$. While the sparse bounds in Theorem \ref{t:SBR} are sharp, up to the boundary, we could not determine sharpness above. 
Whether $\lVert  B_{\lambda}: (p_\lambda, q) \rVert < \infty $ holds for $0 < \lambda < \frac{1}{2}$ and $\frac 4{1+6 \lambda} <  q  \leq 4$ remains an open question. 

\smallskip

Our main result contains Seeger's inequality, as well as some weighted inequalities. While this paper is not the first to show weighted endpoint estimates for the Bochner-Riesz multipliers, the literature on this subject is limited. While Vargas proves that Bochner-Riesz multipliers at the critical index $\frac{n-1}{2}$ are of weak type $(1,1)$ with respect to any weight in the $A_1$-Muckenhoupt class \cite{MR1405057}, weighted inequalities for Bochner-Riesz multipliers below the critical index are new to our knowledge. 
\begin{corollary}\label{c:wtd}  Let  $ 0 < \lambda  < \frac{1}{2}$ and $\rho > \frac4 {4 - 3 p _{\lambda }}$. Then, for every $ w \in A_ 1 \cap RH _{ \rho} $ we have the inequality 
\begin{equation*}
\lVert B _{\lambda }  \rVert _{L^{p_\lambda}(w) \mapsto L ^{p_\lambda, \infty } (w)}  < \infty .   
\end{equation*}
The implied constant depends only on the $A_1$ and reverse H\"{o}lder characteristic of $w$.  
\end{corollary}

The notation above is standard. By \cite{2017arXiv170705212F}*{Thm 1.3}, a  quantitative version of Corollary \ref{c:wtd} is that for $ 0 < \lambda  < \frac{1}{2}$, $\rho > \frac4 {4 - 3 p _{\lambda }}$, and $ w \in A_ 1 \cap RH _{ \rho} $,

\begin{eqnarray*}
\lVert B _{\lambda }  \rVert _{L^{p_\lambda}(w) \mapsto L ^{p_\lambda, \infty } (w)}  \leq c(n, \rho) [w^\rho]_{A_\infty}^{1+\frac{1}{p_\lambda}} \left( [ w]_{A_1} [w]_{RH_\rho} \right)^{\frac{1}{p_\lambda}}. 
\end{eqnarray*}

\smallskip 

The study of sparse bounds for the Bochner-Riesz multipliers was initiated by Benea, Bernicot and Luque \cite{160506401}; their analysis was recently simplified and extended by one of us, Mena, and Reguera 
\cite{170509375}.  The reader should consult that paper for more background and references to the sparse literature.

\section{Lemmas} 

To align our discussion with the notation of \cite{MR1405600}, we expand   Bochner-Riesz operator in the following way. For integers $ j\geq 0$, let 
 $T_j$ be a smooth radial Fourier projection adapted to $\{ |\vec{\xi}|<1  : \textup{dist}( \vec{\xi}, \mathbb{S}^1)  \simeq 2^{-j}\}$. Specifically, pick $\phi \in C^\infty_0([-2, 2])$ such that $\phi \equiv 1$ on $[0, 1]$ and set $\psi (\xi) = \phi(\xi)-\phi(2 \xi)$ along with $\psi_j(\xi) = \psi(2^j \xi)$ for every integer $j \geq 0$. Then $\sum_{j \geq 0} \psi_j(\xi) \equiv 1$ on $[0,1]$, and consequently $(1-|\xi|^2)^\lambda_+ = \sum_{j \geq 0} (1-|\xi|^2)^\lambda_+ \psi_j(1-|\xi|).$ Defining $\widehat{T^\lambda_j f} (\xi)=  2^{\lambda j} (1-|\xi|^2)^\lambda_+ \psi_j(1-|\xi|) \hat{f}(\xi)$, we will have for every $\lambda \in (0,\frac{1}{2})$
 \begin{equation*}
B_{\lambda}  = T^\lambda _0 + \sum_{ j \geq 1} 2^{-\lambda j} T^\lambda_j.
\end{equation*}
 As $|T^\lambda_0 f (x)| \lesssim M_{HL}f(x)$ for all locally integrable functions $f$ and $||M_{HL} : (1,1)|| <\infty$, it suffices to prove Theorem \ref{BP} with $B_\lambda$ replaced by $\sum_{j \geq 1} 2^{-\lambda j }  T^\lambda_j$. We shall need the following inequality concerning the kernel of $T^\lambda_j$.

\begin{lemma}\label{l:tau} Set $ T^\lambda_j f  (x) = \tau^\lambda _j \ast f (x)$. For all integers $ j, N \geq 1$ and $\lambda \in (0, \frac{1}{2})$,  
\begin{equation}\label{e:tau}
\lvert  \tau^\lambda _j (x)\rvert \lesssim_N [ 2 ^{-j} \lvert  x\rvert ] ^{-2N} 2 ^{- \frac{3j}2}, \qquad \lvert  x\rvert \geq 2 ^{j}.    
\end{equation}
While the implied constant depends upon $ N$, it is uniform in $ j \geq 0$ and $\lambda \in [0,1/2]$.  
\end{lemma}

We could not find a source for this estimate, so include a proof here. 

\begin{proof}
Fix $\lambda \in (0,\frac{1}{2})$ and let $\tau_j^\lambda = \tau_j$. Then the Fourier transform of $ \tau _j$, $ \widehat {\tau _j} $, is a radial function, so that it follows from  \cite{MR0304972}*{Chap VII 5.12} that for the radial variable $ u>0$, we have  
\begin{equation*}
\tau _j (u) = 2 \pi \int _{0} ^{\infty }  \widehat {\tau _j} (s) J_0 (2 \pi u s) \;s\,ds 
\end{equation*}
where $ J_0$ is the $ 0$th order Bessel function.  
For an integer $ N$, apply integration by parts $2 N $ times, integrating the Bessel function each time. 
A complication arises here due to the presence of the  Bessel functions, which is addressed in Lemma~\ref{l:IP}.

We have, using a change of variables and \eqref{e:IP}
\begin{align*}
\Bigl\lvert 
\int_0 ^{\infty }  \widehat {\tau _j} (s) J_0 (2 \pi u s) \;s\,ds \Bigr\rvert 
& =  (2 \pi u )^{-2N}
 \Bigl\lvert \int_0 ^{\infty } [ M ^{-1} D M D ] ^{N} \widehat {\tau _j} (s)   \cdot J_0 (2 \pi u s) \;s\,ds 
  \Bigr\rvert . 
\end{align*}
Examine the integral above.   The function  
$
 [ M ^{-1} D M D ] ^{N} \widehat {\tau _j} (s) 
$ 
is supported in an interval close to $ s=1$ of width $ 2 ^{-j}$.  Each derivative applied to $ \widehat {\tau _j} (s) $ gains a factor of $2^j$, so that the function is bounded  bounded by $ C_N 2 ^{2jN}$. 
And, the asymptotics for the Bessel function are 
\begin{equation*}
\lvert  J_{0} (2 \pi u)  \rvert  \lesssim  u ^{- \frac{1}2}  \lesssim  C_N 2 ^{- \frac{j}2}.  
\end{equation*}
Therefore we have the estimate below, which completes the Lemma. 
\begin{equation*}
\lvert  \tau _j (u)\rvert \leq C_N [ u ^{-1} 2 ^{j} ]^{2N}  \cdot  2 ^{- \frac{3j}2}.  
\end{equation*}
\end{proof}

\begin{lemma}\label{l:IP} Let $ \chi $ be a  smooth compactly supported  function on  $(0, \infty  )$. 
We have 
\begin{equation}\label{e:IP}
\int _{0} ^{\infty } \chi (s) J_0 (s) s \;ds = 
- \int [M ^{-1} D M D \chi (s)] J_0 (s) s \;ds, 
\end{equation}
where $ D = \frac d{dx} $ and $ M \phi (s) = s  \phi (s)$ is multiplication by the independent variable.  
\end{lemma}

\begin{proof}
There are two classical identities concerning the Bessel functions,  that we need. Namely, 
\begin{align}\label{e:J0}
D ^{-1} s J_0 (s) &= s J_1 (s), 
\\ \label{e:J1}
D ^{-1} J_1  &= - J_0. 
\end{align}
Using integration by parts,  \eqref{e:J0} and \eqref{e:J1}, in that order we have 
\begin{align*}
\int _{0} ^{\infty } \chi (s) \cdot  sJ_0 (s)  \;ds  & = 
-\int D \chi (s) \cdot s J_1 (s) \; ds 
\\
& = -\int  M D \chi (s) \cdot  J_1 (s) \; ds  
\\
& = -\int D M D \chi (s) \cdot  J_0 (s) \; ds  
\\
&= - \int M ^{-1}  D M D \chi (s) \cdot s J_0 (s) \; ds  . 
\end{align*}
This completes the proof. 
\end{proof}

The $L^p$ mapping properties of $T_j$ are well known, namely 
\begin{equation}\label{e:Tnorm}
\|T_j\|_{p\to p} 
\lesssim 
\begin{cases}
2^{\frac j2}  & p=1 
\\
2 ^{\lambda_p j} & 1< p < \tfrac 43 
\\
 j ^{\frac{1}{4}}  &  p= \tfrac 43
\end{cases}. 
\end{equation}
The first is well known, the second is a consequence of H\"ormander's version of the Carleson-Sj\"olin bounds from \cite{MR0340924}, and the third is due to C\'ordoba \cite{MR544242}. 
These estimates are not sufficient for the endpoint case, however.  Our proof implicitly relies upon the 
 Theorem below, a key inequality of Seeger \cite{MR1405600} .
We will appeal to certain consequences of it.  

\begin{theorem}\cite{MR1405600}*{Thm 2.1}  \label{t:seeger}
Suppose that $1 \leq p< \frac{4}3$ and $ \lambda_p=2(\frac 1 p - \frac 12 ) - \frac 1 2$. 
 Then there is the inequality
\begin{equation}\label{e:seeger}
\Bigl\lVert \sum_{j=1} ^{\infty } T^{\lambda_p}_j f _j  \Bigr\rVert _{p} 
\lesssim \Bigl[\sum_{j=1}  ^{\infty } 2 ^{p\lambda_p j} \lVert f_j\rVert_p ^{p}\Bigr] ^{\frac{1}p}. 
\end{equation}
 
 \end{theorem}
Technically, Theorem 2.1 in \cite{MR1405600} yields \eqref{e:seeger} for the family $\{S_j\}_{j \geq 1}$ where $\widehat{S_j f} (\xi) = \eta(\xi) (1-|\xi|^2)^{\lambda_p}_+ \psi_j(1-|\xi|) \hat{f}(\xi)$ where $\eta \in C^\infty_0(\mathbb{R}^2)$ has $supp (\eta) \subset \{ \xi \in \mathbb{R}^2: |\xi_1/\xi_2| \leq 10^{-1}, \xi_2 >0\}$ and $\psi_j$ is the same function as before. However, rotating the Fourier transform of $\{f_j\}$ yields \eqref{e:seeger} for $\eta$ precomposed with any rotation. It is then straightforward to rewrite for any $j \geq 1$
\begin{eqnarray}
 \psi_j(1-|\xi|) = \sum_{k=1}^K \eta_k(\xi)  \psi_j(1-|\xi|) 
\end{eqnarray}
where $K=O(1)$ and each $\eta_k = \tilde{\eta}_k \circ R_k$ where $R_k$ is a rotation and $\tilde{\eta}_k \in C^\infty_0(\mathbb{R}^2)$ with $supp (\tilde{\eta}_k) \subset \{ \xi \in \mathbb{R}^2: |\xi_1/\xi_2| \leq 10^{-1}, \xi_2 >0\}$. This naturally leads to a decomposition $T^{\lambda_p}_j = \sum_{k=1}^K T^{\lambda_p}_{j,k}$. By the triangle inequality and the validity of \eqref{e:seeger} on the $\{T_{j,k}^{\lambda_p}\}_{j \geq 1}$ for each $1 \leq k \leq K$ yields \eqref{e:seeger} for $\{T_j^{\lambda_p}\}_{j \geq 1}$ as claimed. 
The restriction $1 \leq p < \frac{4}{3}$ above, will by duality, lead to the restriction that $s>4$ in our main theorem.

\section{Proof} 

The method of proof begins with the analysis of Seeger \cite{MR1405600},  and finishes with some additional arguments. 
The analysis depends upon a crucial endpoint estimate  Theorem~\ref{t:seeger} above, and a fine Calder\'on-Zygmund analysis, in the style of Christ \cites{MR1008443,MR796439}.

 To state the main recursive estimate needed in the proof, further set a truncated Bochner-Riesz operator adapted to any dyadic square $Q$ to be 
\begin{equation}\label{e:TQ}
T^{\lambda_p,Q} f =  \sum_{j \;:\;  2^{j} < \ell Q} 2^{-\lambda_p j}   T^{\lambda_p}_j f, \qquad \lambda_p=2(\tfrac 1 p - \tfrac 12 ) - \tfrac 1 2. 
\end{equation}

\begin{lemma}\label{l:recursive} 
Let $1 < p < \frac{4}{3}$,  $4 <q < p'$.  For any square $ E$ and 
pair of measurable and bounded functions  $ f:E \rightarrow \mathbb{C} $ and $h: 3E \rightarrow \mathbb{C}$
there is  a collection $\mathcal{B}$ of disjoint subsquares of $E$ with the properties $\bigl| \bigcup_{Q \in \mathcal{B}}  Q \bigr| \leq \frac{|E|}{100} $ and
\begin{equation}\label{ineq:C2}
\begin{split}
| \langle T^{\lambda_p,E} f, h  \rangle | &\lesssim   \langle f \rangle_{E,p} \langle h \rangle_{3E, q}|E| + 
\sum_{Q \in \mathcal{B}}
  \lvert   \langle T^{\lambda_p,Q} (f 1_Q) ,   \mathbf 1_{3Q}h  \rangle\rvert. 
\end{split}\end{equation}
\end{lemma}

Below, we will cite two inequalities for the Bochner-Riesz operators from the work of Seeger \cite{MR1405600}.  But, our application of the estimates differs from their form in  \cite{MR1405600} in that we will have a truncation  $ T ^{E}$ in place of the Bochner-Riesz operator.  
These new inequalities hold, with the same proof as in \cite{MR1405600}.  This can be seen from the fact that 
the crucial  vector-valued inequality   \cite{MR1405600}*{Thm 2.1} trivially extends to the truncated case. 
And,  all of the endpoint interpolation estimates in \cite{MR1405600}*{(3.5-3.14)} extend uniformly.

\begin{proof}
Fixing $p \in (1, \frac{4}{3})$, let $T_j = T_j^{\lambda_p}$ for all integers $j \geq 1$ and let $T^{Q}=T^{\lambda_p, Q}$ for all dyadic squares $Q$. If $ E$ has area one, the Lemma holds with $ \mathcal B$ being empty.  
Assume therefore that $ E$ has area larger than one.  
The following argument is not symmetric, in that we perform an $ L ^{p}$ 
Calder\'{o}n-Zygmund decomposition on $f$ at level $\langle f \rangle_{E,p}$, while the function $h$ remains untouched. 
This yields $f = g+b$ where $\lVert  g\rVert _{L^\infty} \lesssim \langle f \rangle_{E,p}$ and 
$ \lVert  g\rVert _p \lesssim \lVert  f\rVert _p$.  The function $ b$ satisfies 
$
 b = \sum_{Q \in \mathcal{B}} b_Q,
$ 
where $ \mathcal B$ are disjoint dyadic subsquares of $ E$,  with 
\begin{equation}\label{e:badSmall}
\Bigl\lvert \bigcup_{Q \in \mathcal{B} }Q  \Bigr\rvert \leq \tfrac1{100}{|E|}. 
\end{equation}
The functions $ b_Q = (f - \langle f \rangle_Q) \mathbf 1_{Q}$ are supported on $ Q$, and satisfy  $\lVert b_Q\rVert _p^p \lesssim \langle f \rangle_{E,p}^p |Q| $.   Here, $\langle f\rangle_Q = 
\lvert Q\rvert ^{-1} \int_Q f \;dx$. 

The function $ b$ is then further decomposed into the sum $ \sum_{t =0} ^{\infty } \beta_t$, where the index $ t$ 
indexes the side lengths of the squares $ Q\in \mathcal B$. Namely, 
\begin{equation*}
\beta_0 := \sum_{Q \in \mathcal B \;:\;  \ell Q \leq 0} b_Q, 
\quad \textup{and}  \quad 
\beta_t := \sum_{Q \in \mathcal B \;:\;  \ell Q = 2^t} b_Q,\  0 < t < \log_2 \ell E
\end{equation*}  
where $\ell Q$ and $\ell E$ are the sidelengths of $Q$ and $E$ respectively. 
We then have the equality below for the truncated Bochner-Riesz operator applied to $b$.  
Setting $ 2 ^{1+ j_E} = \ell E$, we have 
\begin{equation*}
T^E b =\sum_{ s \geq 0} \sum_{ j=s+1 }^{ j_E} 2^{-  \lambda_p j} T_j \beta_{j-s} +
T^E  \beta_0  + \sum_{ \sigma \geq 1} \sum_{ j=0}^{ j_E  } 2^{- \lambda_pj } T_j \beta_{j+\sigma}. 
\end{equation*}

The inner product in \eqref{l:recursive} is estimated by 
\begin{equation} \label{e:fourTerms}
\begin{split}
\lvert \langle T^E f,   h \rangle\rvert &\leq  
\lvert \langle T^E g, h \rangle\rvert +
 \sum_{ s \geq 0} \Bigl| \Bigl \langle \sum_{ j=s+1 }^{ j_E} 2^{- \lambda (p)j} T_j \beta_{j-s} ,  h \Bigr \rangle \Bigr| 
\\ & \qquad + \lvert \langle T^E \beta _0,   h \rangle\rvert +\Bigl| \Bigl \langle \sum_{\sigma  \geq 1}  \sum_{ j=0}^{ j_E  }  2^{-  \lambda_pj} T_j \beta_{j+\sigma },  h  \Bigr \rangle \Bigr| 
\\ &=: I + I\!I + I\!I\!I + I\!V. 
\end{split}
\end{equation}
We  estimate  the four terms above in order. 

As $g$ is bounded, the estimate for the first term is trivial. 
\begin{equation} \label{e:I}
\begin{split}
I &=
\lvert \langle  g, T^E  h \rangle\rvert  \lesssim  \langle f \rangle_{E,p} |E|^{1/2}  \lVert  T^E h \rVert_{L^2(\mathbb{R}^2)}\\& \lesssim  \langle f \rangle_{E, p} \langle h \rangle_{3E,2} |E| 
\lesssim  \langle f \rangle_{E, p} \langle h \rangle_{3E,q} |E|. 
\end{split}
\end{equation}
This is matches the first term on the right in \eqref{ineq:C2}.

For term $I\!I$, we turn to the argument of Seeger, and a consequence of Theorerm \ref{t:seeger}.   
Recall from \cite{MR1405600}*{(3.2)},   that for any $ 1 <p < r   < 4/3$, 
\begin{equation*}
\Bigl\| \sum_{ j=s+1 }^{ j_E}2^{ - j \lambda_p} T_j \beta_{j-s}  \Bigr\|_r^r \lesssim 
2^{ - \frac 12 ( \frac{r}{p} -1 ) s} \langle f \rangle_{E, p}^{r-p} \lVert  b\rVert _p^p , \qquad  s \geq 0. 
\end{equation*}
Above, we will take $ r = q' = \frac q{q-1}$.  
By construction,  $\lVert b\rVert _p \lesssim \langle f \rangle_{E,p} |E|^{1/p}$. 
Applying  H\"{o}lder's inequality, we see that 
\begin{equation} \label{e:II}
I\!I \lesssim  2^{-\frac 12 ( \frac{1}{p} - \frac{1}{r} ) s} \langle f \rangle_{E,p}   \langle h \rangle_{3E, q} |E|. 
\end{equation}
For any choice of $r$ the right side  above  is summable over $s$. This again matches the first term on the right in \eqref{ineq:C2}. 

For term $I\!I\!I$, dominate 
\begin{equation*}
 \lvert \langle T ^{E}\beta _0, h \rangle\rvert 
 \leq \sum_{j \geq 0} \lvert \langle 2^{- j \lambda_p }T_j \beta_0 ,  h \rangle\rvert. 
\end{equation*}
Recall from  \cite{MR1405600}*{(3.3)} that for any $p < r < 4/3$ 
\begin{equation*}
\| 2^{- j \lambda_p }T_j \beta_0 \|_r^r \lesssim   2^{- \frac 12 ( \frac{r}{p} -1 )  j} \langle f \rangle_{E,p}^{r-p} \lVert  b\rVert _p^p. 
\end{equation*}
We arrive at the same estimate as in \eqref{e:II}: 
\begin{equation} \label{e:III}
I\!I\!I\lesssim 2^{-\frac 12 ( \frac{1}{p} - \frac{1}{r} ) s} \langle f \rangle_{E,p}   \langle h \rangle_{3E,q} |E|. 
\end{equation}

\smallskip 

The term  $I\!V$ is treated differently. 
Recalling \eqref{e:fourTerms}, estimate $ I\!V$ by first inserting a spatial localization term $ \mathbf 1_{3Q}$ for the inner product associated with $ b_Q$.  
\begin{align}\label{e:IV1}
I\!V & \leq 
 \sum_{\sigma  \geq 1}\Bigl| \Bigl \langle  \sum_{ j=0}^{ j_E  }   \sum_{Q\in \mathcal B \;:\; \ell Q = 2 ^{ j+\sigma }} 
2^{-  \lambda_pj}  \mathbf 1_{(3Q) ^{c}} T_j  b_Q,  h  \Bigr \rangle \Bigr|  
\\
& \qquad + 
\Bigl| \Bigl \langle \sum_{\sigma  \geq 1}  \sum_{ j=0}^{ j_E  }  \sum_{Q\in \mathcal B \;:\; \ell Q = 2 ^{  j+ \sigma  }} 
2^{- j \lambda_p}  \mathbf 1_{3Q} T_j  b_Q,  h  \Bigr \rangle \Bigr|  
\\
& := I\!V_1 + I\!V_2. 
\end{align}

The term $I\!V_1$ is an `error term,' which is dealt with in Lemma~\ref{l:IV1} below. 
In particular, the conclusion \eqref{e:IV1} of that Lemma implies upon summation over $ \sigma \geq 1$ that 
\begin{equation} \label{e:IV1<}
I\!V_1 \lesssim  \langle f \rangle_{E,p}  \langle h \rangle_{3E,q} |E|.
\end{equation}
This matches the bound in \eqref{e:II}.

Turning to $ I\!V_2$,  first recognize that the sum can be reorganized to see that 
\begin{equation*}
I\!V_2 \leq \sum_{Q\in \mathcal B} 
\lvert  \langle  T^Q  b_Q,   \mathbf 1_{3Q}h \rangle\rvert  . 
\end{equation*}
That is, it is close to being the last term in \eqref{ineq:C2}.  
Second,  recall that $b_Q = [ f - \langle f \rangle_Q ] 1_{Q}$. 
Thus, for any square $ Q\in \mathcal B$, 
\begin{align*}
\lvert  \langle  T^Q  b_Q,   \mathbf 1_{3Q}h \rangle\rvert 
& \leq \lvert  \langle f \rangle _{Q} \cdot  \langle  T^Q  \mathbf 1_{Q},   \mathbf 1_{3Q}h \rangle\rvert   
+  \lvert \langle  T^Q  f\mathbf 1_{Q},   \mathbf 1_{3Q}h \rangle\rvert  . 
\end{align*}
The second term on the right is exactly as in \eqref{ineq:C2}. We don't need to argue further about it. 
For the first term on the right, we just use the  $ L ^{p}$-norm bound, H\"{o}lder's inequality, and the stopping condition to see that 
\begin{align*}
\lvert  \langle f \rangle _{Q} \cdot  \langle  T^Q  \mathbf 1_{Q},   \mathbf 1_{3Q}h \rangle\rvert    
& \lesssim \langle f \rangle _{E,p} \lVert  T^Q  \mathbf 1_{Q}\rVert_{q'} \lVert \mathbf 1_{Q}h \rVert _{q} 
\\
& \lesssim \langle f \rangle _{E,p} \lvert  Q\rvert ^{1/q'}  \lVert \mathbf 1_{3Q}h \rVert _{q}. 
\end{align*}
The sum over the disjoint squares $ Q\in \mathcal B$ is clearly seen to be at most 
\begin{equation*}
 \langle f \rangle _{E,p} \sum_{Q\in \mathcal B} 
 \lvert  Q\rvert ^{1/q}  \lVert \mathbf 1_{3Q}h \rVert _{q'} 
 \lesssim  \langle f \rangle _{E,p} \langle h \rangle _{3E,q} \lvert  E\rvert.  
\end{equation*}
The proof is complete. 

\end{proof}

\begin{lemma}\label{l:IV1} 
For every integer $\sigma \geq 1$
\begin{eqnarray}\label{IV1}
 \Bigl| \Bigl \langle  \sum_{ j=0}^{ j_E  } \sum_{Q \in \mathcal{B}_{j+\sigma}}2^{-\lambda_p j} 1_{(3Q)^c} T^{\lambda_p}_j b_Q  ,   h  \Bigr \rangle \Bigr|  \lesssim  2^{-\sigma} \langle f \rangle_{E,p} \langle h \rangle _{3E, q}|E|
\end{eqnarray}
where $1 < p < \frac{4}{3}$,  $4 <q < p'$ and $E$, $ f:E \rightarrow \mathbb{C} $, $h: 3E \rightarrow \mathbb{C}$, $j_E$, $\{b_Q\}_{Q \in \mathcal{B}}$, and $\{\mathcal{B}_t\}_{0 \leq t < \log_2 \ell E}$ are as in Lemma \ref{l:recursive}. 
\end{lemma}

\begin{proof}
Fixing $p \in (1, \frac{4}{3})$, let $T_j = T_j^{\lambda_p}$ for all integers $j \geq 1$. We will show that, uniformly in integers $ j \geq 0$, $ \sigma \geq 1$, $ k \geq 2$, and squares $Q \in \mathcal{B}$
\begin{equation}\label{e:U}
\bigl| \bigl \langle  1_{\Delta ^{k}Q} T_j b_Q, h  \bigr \rangle \bigr|   \lesssim  2^{\lambda_p j} 2^{-k} 2^{-\sigma}   \langle f \rangle _{Q,p} \langle h \rangle_{2^k Q,q}|Q|, 
\end{equation}
where $ \Delta ^{k}Q = 2^k Q \cap (2^{k-1}Q)^c$. 
Summing on $k \geq 2$  and $0 \leq j \leq  j_E$, we have 
\begin{align}
 \Bigl| \Bigl \langle \sum_{j=0}^{j_E} \sum_{Q \in \mathcal{B}_{j + \sigma}} & 2^{-\lambda_p j}1_{(3Q)^c}    T_j b_Q , h  \Bigr \rangle \Bigr|
\\ &\lesssim 2^{-\sigma}  \sum_{k \geq 2}  \sum_{j \geq 0} \sum_{Q \in \mathcal{B}_{j+ \sigma}}   \Bigl| \Bigl \langle  2^{-\lambda_p j} 1_{\Delta ^{k}Q} T_j b_Q, h  \Bigr \rangle \Bigr|  \\& \lesssim
 2^{-\sigma} \sum_{Q \in \mathcal{B}} \langle f \rangle_{Q,p}
 \inf _{x\in Q}[ M \lvert  h\rvert ^{q}] ^{1/q}   |Q| .
\end{align}
An easy application of H\"older's inequality and appealing to the boundedness of the maximal function proves the Lemma.  

\smallskip 

The claim \eqref{e:U} follows from interpolating two separate inequalities.  First, appealing to the kernel estimate \eqref{e:tau}, for any measurable and bounded $F,H : \mathbb{R}^2 \rightarrow \mathbb{C}$
\begin{align}
\bigl| \bigl \langle  1_{\Delta^kQ} T_j (1_QF), H  \bigr \rangle \bigr|
& \leq
\lVert 1_{\Delta^kQ} T_j (1_Q F) \rVert_{L^{\infty}} 
\lVert H 1_{\Delta^kQ}\rVert_1 
\\
& \lesssim _N
 2^{-3j/2}2^{- (N+2)  k} 2^{- N \sigma} \lVert 1_Q F\rVert_1  \lVert H 1_{\Delta^kQ}\rVert_1 
\\  \label{e:11}
& \lesssim 2^{ \frac{j}2}2^{- N k} 2^{- N \sigma}\langle F\rangle_{Q,1} 
\langle H  \rangle _{2^k Q,1}  \lvert  Q\rvert 
, \qquad k \geq 2,\ N > 1. 
\end{align}
The implied constant depends upon $ N$, and below we will specify a choice of $ N = N _{p,q}$.  

We will also appeal to the norm bound \eqref{e:Tnorm} on $ T_j$ on $ L ^{r}$ for any $1<r <4/3$. Namely 
\begin{align}\label{e:111}
\bigl| \bigl \langle  1_{\Delta^kQ} T_j (1_QF), H  \bigr \rangle \bigr|
& \lesssim 
 2^{ \lambda_{r} j} 2^{ \frac{2k}{r^\prime}} \langle   F \rangle  _{Q, r}  \langle  H\rangle _{2^k Q, r^\prime} |Q|.
\end{align}

\smallskip 
As $4 < q < p^\prime$, we may choose $r= q^\prime \left[ 2-\frac{1}{p} - \frac{1}{q} \right]$, interpolate between \eqref{e:11} and \eqref{e:111}, and set $F = b_Q$ and $H=h$ to recover 

\begin{align}
\bigl| \bigl \langle  1_{\Delta^kQ} T_j b_Q, h  \bigr \rangle \bigr|
& \lesssim _N
 2^{ \lambda_{p} j} 2^{- k} 2^{-\sigma}  \langle   b_Q \rangle  _{Q, p}  \langle  h\rangle _{2^k Q, q} |Q|
\end{align}
provided $N =  N_{p,q}\geq  \frac{1}{\frac{1}{p}+\frac{1}{q}-1} \left[ 1 + \frac{2 (2-\frac{1}{p}-\frac{1}{q})}{( q^\prime [ 2 - \frac{1}{p} - \frac{1}{q}])^\prime} \right]$. As $\langle b_Q \rangle_{Q,p} \lesssim \langle f \rangle_{Q,p}$, \eqref{e:U} follows.

 \end{proof}

\begin{proof}[Proof of Theorem~\ref{BP}] 
Let $f,h : \mathbb{R}^2 \rightarrow \mathbb{C}$ be measurable, bounded, and supported on a (translated) dyadic square $E$. 
Indeed, we take the cube $E$ so large, with the supports of $ f$ and $h$ deeply contained inside $ E$, so that we have 
\begin{equation*}
\lvert \langle B _{\lambda_p } f, h \rangle\rvert \leq 
C \langle f \rangle _{E,p} \langle h \rangle _{E,q} \lvert  E\rvert + 
\lvert \langle T^{\lambda_p,E} f, h \rangle\rvert . 
\end{equation*}
This is possible to do, as follows from the kernel estimates in Lemma~\ref{l:tau}. 

We are now in a position to apply the main recursive estimate of Lemma~\ref{l:recursive}. From it, we have 
\begin{equation}
\begin{split}
| \langle T ^{\lambda_p,E} f, h  \rangle | &\leq  C   \langle f \rangle_{E,p} \langle h \rangle_{E, q}|E| 
+ \sum_{Q \in \mathcal{B}}  \lvert   \langle T^{\lambda_p, Q} (f 1_Q) ,   \mathbf 1_{3Q}h  \rangle\rvert .
\end{split}
\end{equation}
The first term on the right we take as the first  contribution to our `sparse' form.  
We then recursively apply this estimate to the individual terms in the sum over $ Q\in \mathcal B$. A straightforward recursion produces a collection of sparse cubes $\tilde{\mathcal{S}}$ for which 
\begin{equation}
| \langle B_{\lambda_p} f, h  \rangle | \lesssim \sum_{Q \in \tilde{S}}   \langle f \rangle_{Q,p} \langle h \rangle_{3Q, q}|Q|.
\end{equation}
Also,  each cube $Q$ there is a set $E_Q \subset Q$ with $\lvert E_Q\rvert \geq \tfrac 1 {100} \lvert Q\rvert $, and the sets $\{E_Q \,:\, Q\in \tilde{\mathcal{S}}$ are pairwise disjoint.   
But this last bound does not quite meet the definition of a sparse bound. But we can replace the term 
$\langle f \rangle_{Q,p}$ by $\langle f \rangle_{3Q,p}$.  We then have the sparse bound.  

\end{proof}

\bibliographystyle{alpha,amsplain}	

\begin{bibdiv}
\begin{biblist}

\bib{160506401}{article}{
      author={Benea, Cristina},
      author={Bernicot, Fr\'ed\'eric},
      author={Luque, Teresa},
       title={Sparse bilinear forms for bochner riesz multipliers and
  applications},
        date={2017},
        ISSN={2052-4986},
     journal={Transactions of the London Mathematical Society},
      volume={4},
      number={1},
       pages={110\ndash 128},
         url={http://dx.doi.org/10.1112/tlm3.12005},
}

\bib{MR0361607}{article}{
      author={Carleson, Lennart},
      author={Sj\"olin, Per},
       title={Oscillatory integrals and a multiplier problem for the disc},
        date={1972},
        ISSN={0039-3223},
     journal={Studia Math.},
      volume={44},
       pages={287\ndash 299. (errata insert)},
        note={Collection of articles honoring the completion by Antoni Zygmund
  of 50 years of scientific activity, III},
      review={\MR{0361607}},
}

\bib{MR796439}{article}{
      author={Christ, Michael},
       title={On almost everywhere convergence of {B}ochner-{R}iesz means in
  higher dimensions},
        date={1985},
        ISSN={0002-9939},
     journal={Proc. Amer. Math. Soc.},
      volume={95},
      number={1},
       pages={16\ndash 20},
      review={\MR{796439}},
}

\bib{MR1008443}{article}{
      author={Christ, Michael},
       title={Weak type endpoint bounds for {B}ochner-{R}iesz multipliers},
        date={1987},
        ISSN={0213-2230},
     journal={Rev. Mat. Iberoamericana},
      volume={3},
      number={1},
       pages={25\ndash 31},
         url={http://dx.doi.org.prx.library.gatech.edu/10.4171/RMI/44},
      review={\MR{1008443}},
}

\bib{MR544242}{article}{
      author={C\'ordoba, A.},
       title={A note on {B}ochner-{R}iesz operators},
        date={1979},
        ISSN={0012-7094},
     journal={Duke Math. J.},
      volume={46},
      number={3},
       pages={505\ndash 511},
  url={http://projecteuclid.org.prx.library.gatech.edu/euclid.dmj/1077313571},
      review={\MR{544242}},
}

\bib{2017arXiv170705212F}{article}{
      author={{Frey}, D.},
      author={{Nieraeth}, B.},
       title={{Weak and strong type $A\_1$-$A\_\infty$ estimates for sparsely
  dominated operators}},
        date={2017-07},
     journal={ArXiv e-prints},
      eprint={1707.05212},
}

\bib{MR0340924}{article}{
      author={H\"ormander, Lars},
       title={Oscillatory integrals and multipliers on {$FL\sp{p}$}},
        date={1973},
        ISSN={0004-2080},
     journal={Ark. Mat.},
      volume={11},
       pages={1\ndash 11},
         url={https://doi.org/10.1007/BF02388505},
      review={\MR{0340924}},
}

\bib{170509375}{article}{
      author={{Lacey}, M.~T.},
      author={{Mena Arias}, D.},
      author={{Reguera}, M.~C.},
       title={{Sparse Bounds for Bochner-Riesz Multipliers}},
        date={2017-05},
     journal={ArXiv e-prints},
      eprint={1705.09375},
}

\bib{MR1405600}{article}{
      author={Seeger, Andreas},
       title={Endpoint inequalities for {B}ochner-{R}iesz multipliers in the
  plane},
        date={1996},
        ISSN={0030-8730},
     journal={Pacific J. Math.},
      volume={174},
      number={2},
       pages={543\ndash 553},
  url={http://projecteuclid.org.prx.library.gatech.edu/euclid.pjm/1102365183},
      review={\MR{1405600}},
}

\bib{MR0304972}{book}{
      author={Stein, Elias~M.},
      author={Weiss, Guido},
       title={Introduction to {F}ourier analysis on {E}uclidean spaces},
   publisher={Princeton University Press, Princeton, N.J.},
        date={1971},
        note={Princeton Mathematical Series, No. 32},
      review={\MR{0304972}},
}

\bib{MR1405057}{article}{
      author={Vargas, Ana~M.},
       title={Weighted weak type {$(1,1)$} bounds for rough operators},
        date={1996},
        ISSN={0024-6107},
     journal={J. London Math. Soc. (2)},
      volume={54},
      number={2},
       pages={297\ndash 310},
         url={https://doi.org/10.1112/jlms/54.2.297},
      review={\MR{1405057}},
}

\end{biblist}
\end{bibdiv}

\end{document}